\documentclass[11pt, reqno]{amsart}
\usepackage{amssymb}
\usepackage{times}
\usepackage{amsmath,amsthm}
\usepackage{amsfonts}
\usepackage{leftidx}
\usepackage{mathrsfs}
\usepackage{enumerate}
\usepackage{abstract}
\usepackage{stmaryrd}
\usepackage{ulem}
\usepackage{graphicx}

\usepackage{float}
\usepackage{hyperref}
\def\R{\mathbb{R}}

\def\d{|\nabla|}

\def\p{\partial}

\def\be{\begin{equation}}
\def\ee{\end{equation}}

\newtheorem{theorem}{Theorem}[section]

\newtheorem{lemma}[theorem]{Lemma}

\theoremstyle{definition}

\theoremstyle{remark}
\newtheorem{remark}[theorem]{Remark}

\numberwithin{equation}{section}
\begin{document}
 \title[Quadratic Schr\"odinger]{Global Solution for the $3D$ quadratic Schr\"odinger equation of $Q(u,\bar{u}$) type }
\author{Xuecheng Wang}
\address{Mathematics Department, Princeton University, Princeton, New Jersey,08544, USA}
\email{xuecheng@math.princeton.edu}
\thanks{}
\maketitle
\begin{abstract}
We study a class of $3D$ quadratic Schr\"odinger equations as follows, $(\p_t -i \Delta) u = Q(u, \bar{u})$.  Different from nonlinearities of the $uu$ type  and the $\bar{u}\bar{u}$ type, which have been studied by Germain-Masmoudi-Shatah in \cite{Germain}, the interaction of $u$ and $\bar{u}$ is very strong  at the low frequency part, e.g., $1\times 1 \rightarrow 0$ type interaction (the size of input frequency is ``$1$'' and the size of output frequency is ``$0$''). It creates a growth mode for the Fourier transform of the profile of solution around a small neighborhood of zero. This growth mode will  again cause   the growth of profile in the medium frequency part due to the $1\times 0\rightarrow 1$ type interaction. The issue of strong $1\times 1\rightarrow 0$ type interaction makes the global existence problem very delicate.

In this paper,   we show that, as long as there are ``$\epsilon$'' derivatives inside the quadratic term $Q (u, \bar{u})$, there exists a global solution for small initial data. As a byproduct,  we also give a simple proof for the almost global existence of the small data solution of $(\p_t -i \Delta)u = |u|^2 = u\bar{u}$, which was first proved by Ginibre-Hayashi \cite{GH}. Instead of using vector fields, we consider this problem purely in Fourier space.
\end{abstract}
\setcounter{tocdepth}{1}

\section{Introduction}
We consider the initial value problem (IVP) as follows,
\begin{equation}\label{mainequation}
\left\{ \begin{array}{ll}
(\p_t - i \Delta) u = Q(u, \bar{u}) & u(t,x): \mathbb{R} \times \mathbb{R}^3\longrightarrow \mathbb{C},\\
\\
u(0)=u_0, & u_0 \in H^{10},\\
\end{array}\right.
\ee
where  the quadratic term $ Q(\cdot, \cdot)$ has ``$\epsilon$'' derivatives in total in the low frequency part, where $0< \epsilon \ll 1$. More precisely, for any $k,k_1,k_2\in \mathbb{Z}$, we assume that  the following estimate holds for  the symbol $q(\xi-\eta, \eta)$ of the quadratic term $Q (u, \bar{u})$,
\[
\| \mathcal{F}^{-1}\big[  q(\xi-\eta, \eta) \psi_{k_1}(\xi-\eta)\psi_{k_2}(\eta)\big] \|_{L^1} + 2^{ k}\|  \mathcal{F}^{-1}\big[ \nabla_\xi q(\xi-\eta, \eta) \psi_{k_1}(\xi-\eta)\psi_{k_2}(\eta)] \|_{L^1} \]
\[+ 2^{\min\{k_1,k_2\}}\|  \mathcal{F}^{-1}\big[ \nabla_\eta q(\xi-\eta, \eta) \psi_{k_1}(\xi-\eta)\psi_{k_2}(\eta)] \|_{L^1}  + 2^{2k} \| \mathcal{F}^{-1}\big[ \nabla_\xi^2 q(\xi-\eta, \eta) \psi_{k_1}(\xi-\eta)\psi_{k_2}(\eta)] \|_{L^1}
\]
\be\label{assumption}
+ 2^{2\min\{k_1,k_2\}} \| \mathcal{F}^{-1}\big[ \nabla_\eta^2 q(\xi-\eta, \eta) \psi_{k_1}(\xi-\eta)\psi_{k_2}(\eta)] \|_{L^1}   \lesssim  2^{\epsilon k_{-}},\quad k_{-}:=\min\{k,0\}.
\ee
 A good example of qudratic term that satisfies (\ref{assumption}) reads  as follows, $Q(u, \bar{u})= \d^\epsilon(|u|^2)/(1+\d^\epsilon)$.

The assumption of no derivatives in the high frequency part is not necessary. The method we use here still works out for the quasilinear case, as long as there are certain symmetries inside (\ref{mainequation}), which help  us to avoid losing derivatives in the energy estimate.  Since we are trying to highlight the $L^\infty$ decay estimate part, we use this assumption to make energy estimate easier. 

There is a very large literature on the small data global existence results of nonlinear Schr\"odinger equations. By using the energy estimates and the decay estimate of the linear solution, one can prove global existence for quadratic nonlinearities in dimension $4$ and higher, see, e.g., Klainermann-Ponce \cite{Klainerman1} and Strauss \cite{Strauss}. This method also works in dimension $3$ if the order of nonlinearities  is strictly greater than $2$, which is also known as a Strauss exponent in dimension $3$. 

The question of small data global existence (SDGE) for $3D$ quadratic Schr\"odinger is more subtle. On the one hand, it   depends on the type of nonlinearity. On the other hand, it also depends on the decay rate of initial data as $|x|\rightarrow \infty$. As mentioned in the abstract, SDGE is known for nonlinearities of type $uu$ or $\bar{u}\bar{u}$ or any combination of them, see Germain-Masmoudi-Shatah \cite{Germain}. For the gauge-invariant nonlinearity of  $|u|u$ type, one can also obtain global existence, see Cazenave-Weissler \cite{Cazenave}. By using the vector fields method, Ginibre and Hayashi \cite{GH} showed the almost global existence for small initial data for nonlinearities of $u\bar{u}$ type.

Although it is still not clear whether SDGE is true for the nonlinearities of type $u\bar{u}$, It is certainly clear that the answer will depend on in what sense the initial data is small. For  3D quadratic Schr\"odinger with the nonlinearity  $u \bar{u}$, Ikeda-Inui \cite{II} showed that actually  the solution blows up in polynomial time for a class of small  $L^2$ initial data, which decays at  rate $\displaystyle{\frac{1}{|x|^{2-\epsilon}}}$ as $|x|\rightarrow\infty$, where $0< \epsilon<1/2.$ Therefore,  for the validity of SDGE for general nonlinearities of $u\bar{u}$ type, initial data should decay faster than $\displaystyle\frac{1}{|x|^{2-\epsilon}}$ for all $\epsilon >0$. 

In this paper, we are trying to improve the understanding of this problem. We show that solution of (\ref{mainequation}) exists globally if there are ``$\epsilon$'' derivatives at the low frequency part of nonlinearity (i.e., (\ref{assumption})) and the initial data decays faster than $\displaystyle\frac{1}{|x|^{2+\gamma}}$ for any $\gamma$ such that $0< \gamma < \epsilon$. Before stating our main theorem, we define the $Z$-normed space as follows,
\be\label{definitionZ}
\| f\|_{Z}:= \sup_{k\in \mathbb{Z}} 2^{-\gamma  k +2k_{+}}\big( 2^{-k/2   } \| P_k f \|_{L^2} + 2^{ k/2  } \| \nabla_\xi \widehat{f}(\xi)\psi_k(\xi) \|_{L^2}+ 2^{3k/2 }\| \nabla_\xi^2 \widehat{f}(\xi)\psi_k(\xi) \|_{L^2}\big).
\ee

Our main theorem is stated as follows,
\begin{theorem}\label{maintheorem}
Fix $\epsilon$ and $\gamma$, where $0< \epsilon \ll 1$ and $\gamma < \epsilon$. Assume that the initial data $u_0 $ satisfies the following assumption, 
\be
\| u_0\|_{H^{10}} + \|  u_0\|_{Z} \leq \epsilon_0,
\ee
where $\epsilon_0$ is a sufficiently small constant. Then there exists a unique global solution for the IVP   \textup{(\ref{mainequation})} and the solution posses scattering property.  Moreover, the following estimate  holds,
\be
\sup_{t\in[0, \infty)} \| u(t)\|_{H^{10}} + \| e^{-it \Delta}  u(t)\|_{Z} + (1+t)^{1+\gamma/2} \| u(t)\|_{L^\infty} \lesssim  \epsilon_0.
\ee
\end{theorem}
 
\section{Preliminary}
For any two numbers $A$ and $B$, we use notation $A\lesssim B$ and $A\ll B$ to denote  $A\leq C B$ and $A\leq c B$ respectively, where $C$ is an absolute constant and $c$ is a sufficiently small constant.  For an integer $k\in\mathbb{Z}$, we use $k_ {+} $ to denote $\max\{k,0\} $ and we use $k_ {-} $ to denote $\min\{k,0\}$. For any $k\in\mathbb{Z}$, we use ``$f_k$'' to abbreviate $P_k f$, where $P_k$ is the Littlewood-Paley operator.

 Throughout the proof of theorem \ref{maintheorem}, we will use the following bilinear estimate and the $L^\infty$ decay estimate constantly. 
 \begin{lemma}\label{bilinearlemma}
For $1\leq p, q, r\leq \infty$, $f \in L^p(\R^3)$ and $g\in L^q(\R^3)$, the following bilinear estimate holds,
\be\label{bilinearestimate}
\| \mathcal{F}^{-1}[\int_{\R^3} m(\xi, \eta) \widehat{f}(\xi-\eta) \widehat{g}(\eta) d \eta  ]\|_{L^r}\lesssim \| \mathcal{F}^{-1}[m(\xi-\eta, \eta)]\|_{L^1} \| f\|_{L^p} \| g \|_{L^q},
\ee
if $ 1/r=1/p + 1/q$.
 \end{lemma}
\begin{lemma}\label{decaylemma}
 The following $L^\infty$ decay estimate holds for a function in the $Z$-normed space,
 \be\label{decayestimate}
 \| e^{ it \Delta} P_k f\|_{L^\infty}\lesssim t^{-3/2} \| P_k f \|_{L^1}.
 \ee
\end{lemma}
 \begin{proof}
 See \cite{Guoz}[Theorem 1].
\end{proof}

\section{Proof of the main theorem}
\subsection{The set-up}  The bootstrap assumption is stated as follows, 
\be\label{smallness}
\sup_{t\in[0,T]}   \| u(t)\|_{H^{10}} + \| e^{-it\Delta} u(t)\|_{Z} \lesssim \epsilon_1:=\epsilon_0^{5/6}. 
\ee

Define the profile `` $f(t)$ ''of the solution $u(t)$ as $f(t):= e^{-it \Delta} u(t)$. From (\ref{decayestimate}) in Lemma \ref{decaylemma} and the definition of $Z$-norm in (\ref{definitionZ}), the following estimate holds  for any $t\in [2^{m-1}, 2^m]$,
\be\label{eqn201}
\| e^{ it \Delta} f_{k}(t)\|_{L^\infty} \lesssim 2^{-3m/2 -k -2k_{ +}+\gamma k }\epsilon_1,\quad \sum_{k\in \mathbb{Z}}\| e^{ it \Delta} f_{k}(t)\|_{L^\infty} \lesssim 2^{-m-\gamma m/2}\epsilon_1,
\ee
\be\label{eqn200}
 \| f_k(t)\|_{L^2}   \lesssim 2^{k/2+\gamma k-10k_+} \epsilon_1, \quad 2^{k/2}\| \nabla_\xi\widehat{f}_k(t,\xi)\|_{L^2} +  2^{3k/2}\| \nabla_\xi^2 \widehat{f}_k(t,\xi)\|_{L^2}\lesssim 2^{ \gamma k -2k{+}}\epsilon_1.
\ee 
Define
\[
\chi_k^1:=\{(k_1,k_2): |k_1-k_2|\leq 10, k  \leq k_1+10\},\quad \chi_k^2:=  \{(k_1,k_2): |k-k_1|\leq 10, k_2\leq k_1-10\}.
\]
  From (\ref{mainequation}), the following equality holds for any $t_1, t_2\in[2^{m-1}, 2^m]$,
\[
\widehat{f}(t_2, \xi)\psi_{k}(\xi)- \widehat{f}(t_1, \xi)\psi_{k}(\xi) =\sum_{(k_1,k_2)\in \chi_k^1\cup \chi_k^2}  H_{k,k_1,k_2}^m(\xi),\quad  H_{k,k_1,k_2}^m(\xi):=\int_{t_1}^{t_2} \int_{\R^3} \psi_k(\xi) \big[ q(\xi-\eta, \eta) 
\]
\be\label{duhamel1}
\times e^{i s (|\xi|^2 - |\xi-\eta|^2 +|\eta|^2)} \widehat{f_{k_1}}(t, \xi-\eta)\widehat{\bar{f}_{k_2}}(t, \eta)  +q(\eta,\xi-\eta ) e^{i s (|\xi|^2 +|\xi-\eta|^2 -|\eta|^2)} \widehat{\bar{f}_{k_1}}(t, \xi-\eta)\widehat{ {f}_{k_2}}(t, \eta) \big] d\eta d s  . 
\ee

After applying $\nabla_\xi$ once and twice to (\ref{duhamel1}), we have the following eqaulity for any fixed $l,n\in \{1,2\},$
\be\label{eqn20}
\p_{\xi_l} \widehat{f}(t_2, \xi)\psi_{k}(\xi)- \p_{\xi_l} \widehat{f}(t_1, \xi)\psi_{k}(\xi) = \sum_{(k_1,k_2)\in \chi_k^1\cup \chi_k^2} \sum_{i=1,2  }I_{k,k_1,k_2}^{  m,i} + \sum_{(k_1,k_2)\in \chi_k^2} I_{k,k_1,k_2}^{ m,3},
\ee
\be\label{eqn1}
  \p_{\xi_l}\p_{\xi_n}  \widehat{f}(t_2, \xi)\psi_{k}(\xi) -  \p_{\xi_l}\p_{\xi_n} \widehat{f}(t_1, \xi)\psi_{k}(\xi) = \sum_{(k_1,k_2)\in \chi_k^1\cup \chi_k^2}  \sum_{i=1,2,3}J_{k,k_1,k_2}^{ m,1,i} +  \sum_{(k_1,k_2)\in \chi_k^2} \sum_{i=1,2,3} J_{k,k_1,k_2}^{ m,2,i},
\ee
where  
\be\label{eqn2}
I_{k,k_1,k_2}^{ m,1}(\xi) :=\int_{t_1}^{t_2} \int_{\R^3} e^{i 2s \xi \cdot\eta}  \p_{\xi_l}\big(q(\xi-\eta, \eta)\widehat{f_{k_1}}(t, \xi-\eta)\big)\widehat{\bar{f}_{k_2}}(t, \eta)  \psi_k(\xi) d\eta d s,
\ee
\be\label{eqn3}
I_{k,k_1,k_2}^{ m,2}(\xi)  :=\int_{t_1}^{t_2} \int_{\R^3} e^{i 2s \xi \cdot\eta} 2 i  s \eta_l  q(\xi-\eta, \eta) \widehat{f_{k_1}}(t, \xi-\eta) \widehat{\bar{f}_{k_2}}(t, \eta)  \psi_k(\xi) d\eta d s,
\ee
\[
I_{k,k_1,k_2}^{ m,3}(\xi):=\int_{t_1}^{t_2} \int_{\R^3} e^{i 2s \xi \cdot(\xi-\eta)}  \p_{\xi_l}\big( q(\eta,\xi-\eta )   \widehat{\bar{f}_{k_1}}(t, \xi-\eta)\big)\widehat{  {f}_{k_2}}(t, \eta)  \psi_k(\xi) d\eta d s
\]
\be\label{eqn240}
+ \int_{t_1}^{t_2} \int_{\R^3} e^{i 2s \xi \cdot(\xi-\eta) }i2 s (2\xi_l-\eta_l)  q(\eta,\xi-\eta )   \widehat{\bar{f}_{k_1}}(t, \xi-\eta) \widehat{  {f}_{k_2}}(t, \eta)  \psi_k(\xi) d\eta d s,
\ee
\be\label{eqn4}
J_{k,k_1,k_2}^{ m,1,1}(\xi) :=\int_{t_1}^{t_2} \int_{\R^3} e^{i 2s \xi \cdot\eta}   \p_{\xi_l}\p_{\xi_n}  \big(  q(\xi-\eta, \eta) \widehat{f_{k_1}}(t, \xi-\eta) \big)\widehat{\bar{f}_{k_2}}(t, \eta) \psi_k(\xi) d\eta d s,
\ee
\be\label{eqn5}
J_{k,k_1,k_2}^{ m,1,2}(\xi) :=\sum_{(l_1,n_1) =(l,n),(n,l) }\int_{t_1}^{t_2} \int_{\R^3} e^{i 2s \xi \cdot\eta} 2i  s \eta_{n_1} \p_{\xi_{l_1}}  \big( q(\xi-\eta, \eta) \widehat{f_{k_1}}(t, \xi-\eta) \big) \widehat{\bar{f}_{k_2}}(t, \eta) \psi_k(\xi) d\eta d s,
\ee
\be\label{eqn6}
J_{k,k_1,k_2}^{ m,1,3}(\xi):=-\int_{t_1}^{t_2} \int_{\R^3} e^{i 2s \xi \cdot\eta}  4 s^2 \eta_l \eta_n  q(\xi-\eta, \eta) \widehat{f_{k_1}}(t, \xi-\eta) \widehat{\bar{f}_{k_2}}(t, \eta)  \psi_k(\xi) d\eta d s,
\ee
\be\label{eqn271}
J_{k,k_1,k_2}^{ m,2,1}(\xi):=\int_{t_1}^{t_2} \int_{\R^3} e^{i 2s \xi \cdot(\xi-\eta)}   \p_{\xi_l}\p_{\xi_n}  \big(  q(\eta,\xi-\eta) \widehat{\bar{f}_{k_1}}(t, \xi-\eta) \big)   \widehat{{f}_{k_2}}(t, \eta) \psi_k(\xi) d\eta d s,
\ee
\[
J_{k,k_1,k_2}^{ m,2,2}(\xi):=\sum_{(l_1,n_1) =(l,n),(n,l) }\int_{t_1}^{t_2} \int_{\R^3} e^{i 2s \xi \cdot(\xi-\eta)} 2i  s (2\xi_{n_1}-\eta_{n_1}) \]
\be\label{eqn272}
\times \p_{\xi_{l_1}}  \big( q(\eta,\xi-\eta) \widehat{\bar{f}_{k_1}}(t, \xi-\eta) \big)\widehat{{f}_{k_2}}(t, \eta) \psi_k(\xi) d\eta d s,
\ee
\be\label{eqn274}
J_{k,k_1,k_2}^{ m,2,3}(\xi):=- \int_{t_1}^{t_2} \int_{\R^3} e^{i 2s \xi \cdot(\xi-\eta)} 4  s (2\xi_{l }-\eta_{l })(2\xi_{n} -\eta_n)  q(\eta,\xi-\eta) \widehat{\bar{f}_{k_1}}(t, \xi-\eta)   \widehat{{f}_{k_2}}(t, \eta) \psi_k(\xi) d\eta d s.
\ee
\subsection{The growth of $Z$-norm and energy over time}

\begin{lemma}\label{L2derivatieveestimate}
Under the bootstrap assumption \textup{(\ref{smallness})}, the following estimates hold, 
\be\label{eqn298}
  \|H_{k,k_1,k_2}^m  \|_{L^2} \lesssim\min\{2^{m +3 \min\{k,k_2\}/2  + k_1/2  } ,2^{-m/2-k_1    }\} 2^{\epsilon k_{-}+ k_2/2+\gamma(k_1+k_2)-2(k_{1,+}+k_{2,+})}\epsilon_1^2,
\ee
\be\label{eqn230}
\|    \p_t f_k  \|_{L^2} \lesssim   \min\{2^{-m + k/2-3k_{+}+\epsilon k_{-}} ,2^{-(1+\gamma/2)m -10k_{+} +\epsilon k_{-}}\}\epsilon_1^2,
\ee
\be\label{e300}
 \| e^{ it \Delta}    \p_t  f_k  \|_{L^\infty}\lesssim 2^{-5m/2-k+\gamma k-2k_{+}} \epsilon_1^2.
\ee
 
\end{lemma}
\begin{proof}
Recall   (\ref{duhamel1}). 
 From the $L^2-L^\infty$ type bilinear estimate (\ref{bilinearestimate}) in Lemma \ref{bilinearlemma}, (\ref{assumption}), and the $L^\infty\rightarrow L^2$ type   Sobolev embedding  or  alternatively   the $L^2 \rightarrow L^1$ type Sobolev embedding and $L^2-L^2$ type bilinear estimate,  the following estimates hold, 
\[
  \|H_{k,k_1,k_2}  \|_{L^2} \lesssim \sup_{t\in[2^{m-1}, 2^m]} 2^{m +3 \min\{k,k_2\}/2 +\epsilon k_{-} } \| f_{k_1}(t)\|_{L^2}  \| f_{k_2}(t)\|_{L^2}\]
\be\label{eqn301} 
\lesssim 2^{\epsilon k_{-}+\gamma(k_1+k_2)} 2^{m +3 \min\{k,k_2\}/2    +k_1/2 +k_2/2-10 k_{1,+} -10 k_{2,+}}\epsilon_1^2.
\ee
\[
  \|H_{k,k_1,k_2}  \|_{L^2} \lesssim  \sup_{t\in[2^{m-1}, 2^m]} 2^{m+\epsilon k_{-} } \| e^{ it\Delta} f_{k_1}\|_{L^\infty} \| f_{k_2}(t)\|_{L^2} 
\]
\be\label{eqn300}
\lesssim 2^{ \epsilon k_{-}+\gamma(k_1+k_2)} 2^{-m/2-k_1    +k_2/2-2k_{1,+}-2k_{2,+}} \epsilon_1^2.
\ee
Combine estimates (\ref{eqn301}) and (\ref{eqn300}), it is easy to see our desired estimate (\ref{eqn298})
 holds. 

  Very similarly, from the $L^2\rightarrow L^{3/2}$ type Sobolev embedding, the $L^2-L^{6 }$ type bilinear estimate, and the $L^2-L^{\infty }$ type bilinear estimate,  the following estimates hold,
\[
\|    \p_t f_k  \|_{L^2}\lesssim \sum_{ k_2\leq k_1 } \min\big\{2^{ k/2+\epsilon k_{-}} \| e^{ it \Delta} f_{k_1}\|_{L^{6 }} \| f_{k_2}\|_{L^2}, 2^{\epsilon k_{-}}\| e^{ it \Delta} f_{k_2}\|_{L^\infty} \| f_{k_1}\|_{L^2} \big\}\]
\[ \lesssim  \sum_{k_2\leq k_1 } \min\big\{2^{k/2 +\epsilon k_{-} - m- k_1/2-3k_{1,+}   + k_2/2+\gamma k_2-2k_{2,+}}\epsilon_1^2,\quad 2^{\epsilon k_{-}-10 k_{1,+}-2k_{2,+}} \]
\[\times\min\{2^{-3m/2-k_2+\gamma k_2}, 2^{2k_2 +\gamma k_2} \}\epsilon_1^2\big\}
 \lesssim \min\{2^{-m + k/2+\epsilon k_{-}-3k_{+}} , 2^{-(1+\gamma/2)m -10k_{+} +\epsilon k_{-}}\}\epsilon_1^2. \]
\[ \| e^{ it \Delta}    \p_t  f_k  \|_{L^\infty}\lesssim \sum_{k_2\leq k_1}  2^{\epsilon k_{-}}\| e^{ it \Delta} f_{k_1}\|_{L^\infty}\| e^{ it \Delta} f_{k_2}\|_{L^\infty}  \lesssim  \sum_{k_1\geq k-10, k_2\leq -m/2}  2^{-3m/2-k_1+\gamma k_{1}-2k_{1,+}}\]
\[\times 2^{ 2k_2+\gamma k_2+\epsilon k_{ -} } \epsilon_1^2+    \sum_{k_1\geq k-10, k_2\geq-m/2}  2^{-3m-k_1+\gamma k_1-k_2-2k_{1,+}-2k_{2,+}+ \epsilon k_{-} }  \epsilon_1^2\lesssim 2^{-5m/2-k+\gamma k-2k_{+}} \epsilon_1^2.
\]
Now it is easy to see our desired estimates (\ref{eqn230}) and (\ref{e300}) hold. 
\end{proof}
\begin{lemma}\label{L2weightone}
Under the bootstrap assumption \textup{(\ref{smallness})}, the following estimates hold,
\[
 \sum_{i=1,2}    2^{k/2}  \| 	I_{k,k_1,k_2}^{m,i}(\xi) \|_{L^2}  +  \sum_{i=1,2,3}  2^{3k/2}  \| J_{k,k_1,k_2}^{ m,1,i}(\xi) \|_{L^2} 
\]
\be\label{eqn410}
 \lesssim 2^{ -k/2 +k_1/2+ \epsilon  k_{-}+\gamma(k_1+k_2)-2(k_{1,+}+k_{2,+})} \min\{2^{m+3\min\{k,k_2\}/2 +k_2/2}, 2^{-m/2-k_2}\}\epsilon_1^2. 
 \ee
\end{lemma}
\begin{proof}
We first estimate $ I_{k,k_1,k_2}^{m,1}(\xi)$ and $J_{k,k_1,k_2}^{m,1,1}(\xi)$.
Recall (\ref{eqn2}), and (\ref{eqn4}). 
From the $L^2-L^\infty$ type bilinear estimate and the $L^\infty\rightarrow L^2$ type Sobolev embedding,  or  alternatively   the $L^2 \rightarrow L^1$ type Sobolev embedding and $L^2-L^2$ type bilinear estimate, the following estimate holds, 
\[
 \| 	I_{k,k_1,k_2}^{m,1}(\xi) \|_{L^2}  \lesssim \sup_{t\in[2^{m-1}, 2^m]} 2^{m  + \epsilon k_{ -}}  \big(2^{-k} \| f_{k_1}(t)\|_{L^2} + \| \nabla_\xi \widehat{f_{k_1}}(t, \xi)\|_{L^2}  \big)\min\{   2^{3\min\{k,k_2\}/2}\| f_{k_2}(t)\|_{L^2}   ,
\]
\be\label{eqn320}
  \| e^{ it \Delta} f_{k_2}(t)\|_{L^\infty}\}  \lesssim 2^{-k+k_1/2 +\epsilon  k_{-}+\gamma(k_1+k_2)-2(k_{1,+}+k_{2,+}) } \min\{2^{m+3\min\{k,k_2\}/2 +k_2/2}, 2^{-m/2-k_2}\}\epsilon_1^2.
\ee
\[
 \| 	J_{k,k_1,k_2}^{ m,1,1}(\xi) \|_{L^2}  \lesssim \sup_{t\in[2^{m-1}, 2^m]} 2^{m  + \epsilon k_{ -}}   \min\{  \| e^{ it \Delta} f_{k_2}(t)\|_{L^\infty},2^{3\min\{k,k_2\}/2}\| f_{k_2}(t)\|_{L^2}  \} 
\]
\[
\times  \big(2^{-2k} \| f_{k_1}(t)\|_{L^2} +2^{-k } \| \nabla_\xi \widehat{f_{k_1}}(t, \xi)\|_{L^2}  + \| \nabla_\xi^2 \widehat{f_{k_1}}(t, \xi)\|_{L^2}   \big)
\]
\be\label{eqn390}
 \lesssim 2^{ -2k + k_1/2 +\epsilon  k_{-}+\gamma(k_1+k_2)-2(k_{1,+}+k_{2,+})} \min\{2^{m+3\min\{k,k_2\}/2 +k_2/2}, 2^{-m/2-k_2}\}\epsilon_1^2.
\ee

Now, we proceed to estimate  $ I_{k,k_1,k_2}^{m,2}(\xi)$ and $J_{k,k_1,k_2}^{m,1,i}(\xi)$, $i\in\{2,3\}$.
Recall (\ref{eqn3}),(\ref{eqn5}), and (\ref{eqn6}).
We do integration by parts in ``$\eta$'' once  for $I_ {k_1, k_2} ^ {m, 2} $ and $J_ {k_1, k_2}^{m, 1,2} $ and   do integration by parts in ``$\eta$'' twice for $J_ {k_1, k_2}^{m, 1,3} $. As a result, we have
\be\label{eqn340}
I_{k,k_1,k_2}^{m,2}(\xi) :=\int_{t_1}^{t_2} \int_{\R^3} e^{i 2s \xi \cdot\eta} \frac{-\xi}{|\xi|^2} \cdot  \nabla_\eta\big(  \eta_l  q(\xi-\eta, \eta) \widehat{f_{k_1}}(t, \xi-\eta) \widehat{\bar{f}_{k_2}}(t, \eta) \big)\psi_k(\xi) d\eta d s,
\ee
\[
J_{k,k_1,k_2}^{ m,1,2}(\xi):= \sum_{(l_1,n_1)=(l,n),(n,l) }\int_{t_1}^{t_2} \int_{\R^3} e^{i 2s \xi \cdot\eta}\psi_k(\xi)\]
\[\times  \frac{-\xi}{|\xi|^2} \cdot  \nabla_\eta\Big(   \eta_{n_1} \p_{\xi_{l_1}}  \big( q(\xi-\eta, \eta) \widehat{f_{k_1}}(t, \xi-\eta) \big) \widehat{\bar{f}_{k_2}}(t, \eta) \Big) d\eta d s,
\]
\[
J_{k,k_1,k_2}^{ m,1,3}(\xi):= \int_{t_1}^{t_2} \int_{\R^3} e^{i 2s \xi \cdot\eta}   \frac{  \xi}{|\xi|^2} \cdot \nabla_\eta \cdot \Big(   \frac{\xi}{|\xi|^2} \cdot \nabla_\eta \big( \eta_l \eta_n  q(\xi-\eta, \eta) \widehat{f_{k_1}}(t, \xi-\eta) \widehat{\bar{f}_{k_2}}(t, \eta)\big) \Big) \psi_k(\xi) d\eta d s.
\]
  From the $L^2-L^\infty$ type estimate and the $L^\infty \longrightarrow L^2$ type Sobolev embedding, we have
\[
 \|I_{k,k_1,k_2}^{m,2}(\xi) \|_{L^2} \lesssim  \sup_{t\in[2^{m-1}, 2^m]} 2^{m+\epsilon k_{-}+k_2-k} \big[ \| \nabla_\xi \widehat{f_{k_1}}(t,\xi)\|_{L^2}  \| e^{ it \Delta} f_{k_2}(t)\|_{L^\infty}       + \| e^{ it \Delta} f_{k_1}(t)\|_{L^\infty} 
\] 
\be\label{eqn401}
\times \big( \| \nabla_\xi \widehat{f_{k_2}}(t,\xi)\|_{L^2} + 2^{-k_2} \|   {f_{k_2}}(t)\|_{L^2} \big) \big] \lesssim 2^{-m/2 -k_1/2 -k +\epsilon  k_{-}+\gamma(k_1+k_2)-2(k_{1,+}+k_{2,+}) } \epsilon_1^2,
\ee
\[
 \|I_{k,k_1,k_2}^{m,2}(\xi) \|_{L^2} \lesssim  \sup_{t\in[2^{m-1}, 2^m]} 2^{m +k_2-k+\epsilon k_{-}} 2^{3\min\{k,k_2\}/2}\big[ \| \nabla_\xi \widehat{f}_{k_1}(t,\xi)\|_{L^2} \| f_{k_2}(t)\|_{L^2}
 +
  \| f_{k_1}(t)\|_{L^2}  
\]
\be\label{eqn402}
\times  \big(\| \nabla_\xi \widehat{f}_{k_2}(t,\xi)\|_{L^2} + 2^{-k_2}\| f_{k_2}(t)\|_{L^2}\big)\big]\lesssim 2^{m+k_2/2+k_1/2-k +3\min\{k,k_2\}/2 +\epsilon  k_{-}+\gamma(k_1+k_2)-2(k_{1,+}+k_{2,+})}\epsilon_1^2.
\ee

Very similarly, from the $L^2-L^\infty$ type and $L^4-L^4$ type bilinear estimates, we have
\[
   \|J_{k,k_1,k_2}^{m,1,2}(\xi)\|_{L^2}\lesssim \sup_{t\in[2^{m-1}, 2^m]}
2^{m+\epsilon k_{-}  +k_2-k}\big[  \| \nabla_\xi^2 \widehat{f}_{k_1}(t, \xi)\|_{L^2} \| e^{ it\Delta} f_{k_2}(t)\|_{L^\infty} +2^{-\min\{k,k_2\}}
\]
\[
\times \| e^{ it \Delta} \mathcal{F}^{-1}[\nabla_\xi  \widehat{f}_{k_1}(t, \xi)]\|_{L^4} \| e^{ it\Delta} f_{k_2}(t)\|_{L^4} + 2^{-k_2-k} \| e^{ it \Delta} g_{k_1}\|_{L^2} \| g_{k_2}(t)\|_{L^2}
\]
\[
  +  \| e^{ it \Delta} \mathcal{F}^{-1}[\nabla_\xi  \widehat{f}_{k_2}(t, \xi)]\|_{L^4}  \big(  \| e^{ it \Delta} \mathcal{F}^{-1}[\nabla_\xi  \widehat{f}_{k_1}(t, \xi)]\|_{L^4}  + 2^{-k}\|e^{ it \Delta} f_{k_1}(t) \|_{L^4}\big) \big]
\]
\be\label{eqn420}
\lesssim 2^{-m/2  -2k-k_2/4- k_1/4+\epsilon  k_{-}+\gamma(k_1+k_2)-2(k_{1,+}+k_{2,+})}\epsilon_1^2,
\ee

\[
  \|J_{k,k_1,k_2}^{ m,1,3}(\xi)\|_{L^2}\lesssim \sup_{t\in[2^{m-1}, 2^m]}
2^{m+\epsilon k_{-}  +2k_2-2k}\big[  (\| \nabla_\xi^2 \widehat{f}_{k_1}(t, \xi)\|_{L^2}  + 2^{-k_2} \|  \nabla_\xi  \widehat{f}_{k_1}(t, \xi) \|_{L^2})
\]
\[
\times \| e^{ it\Delta} f_{k_2}(t)\|_{L^\infty}   + \| e^{ it\Delta} f_{k_1}(t)\|_{L^\infty} 
  \big(\|   \nabla_\xi^2  \widehat{f}_{k_2}(t, \xi) \|_{L^2}+ 2^{-k_2}\|   \nabla_\xi  \widehat{f}_{k_2}(t, \xi) \|_{L^2} + 2^{-2k_2} \| f_{k_2}(t)\|_{L^2})  
\]
\[
+ \| e^{ it \Delta} \mathcal{F}^{-1}[\nabla_\xi  \widehat{f}_{k_1}(t, \xi)]\|_{L^4}  \| e^{ it \Delta} \mathcal{F}^{-1}[\nabla_\xi  \widehat{f}_{k_2}(t, \xi)]\|_{L^4} \big]
\]
\be\label{eqn423}
\lesssim  2^{-m/2-2k -k_1/2 +  \epsilon  k_{-}+\gamma(k_1+k_2)-2(k_{1,+}+k_{2,+})}\epsilon_1^2.
\ee
In above estimates, we used the fact that,
\be\label{L4estimate1}
\| e^{ it\Delta} f_{k_2}(t)\|_{L^4}  \lesssim \| e^{ it\Delta} f_{k_2}(t)\|_{L^\infty}^{1/2}  \|   f_{k_2}(t)\|_{L^2}^{1/2} \lesssim 2^{\gamma k_2 -2k_{2,+}}\min\{2^{-3m/4-k_2/4} , 2^{5k_2/4}\}\epsilon_1.
\ee
\be\label{L4estimate2}
\| e^{ it\Delta} \mathcal{F}^{-1}[\nabla_\xi \widehat{f_{k_2}}(t,\xi)]\|_{L^4}  \lesssim 2^{\gamma k_2 -2k_{2,+} }\min\{2^{-3m/4-5k_2/4}, 2^{ k_2/4} \}\epsilon_1.
\ee

On the other hand, if we use the volume of the support of $\xi$ or $\eta$   first and  then use the   bilnear estimate (\ref{bilinearestimate}) in Lemma \ref{bilinearlemma},  the following estimates hold,
\[
  \|J_{k,k_1,k_2}^{m,1,2}(\xi)\|_{L^2}\lesssim \sup_{t\in[2^{m-1}, 2^m]} 2^{m + \epsilon k_{-} + k_2- k+ 3\min\{k,k_2\}/2}\big[\big( 2^{-\min\{k,k_2\}} \| \nabla_\xi   \widehat{f}_{k_1}(t, \xi)\|_{L^2}\]
  \[+\| \nabla_\xi^2 \widehat{f}_{k_1}(t, \xi)\|_{L^2}   \big) \|f_{k_2}(t)\|_{L^2} +  2^{-k} \|  f_{k_1}(t)\|_{L^2} \big(  \| \nabla_\xi  \widehat{f}_{k_2}(t, \xi)\|_{L^2} +2^{-k_2} \| f_{k_2}(t)\|_{L^2}\big)    +  \| \nabla_\xi  \widehat{f}_{k_1}(t, \xi)\|_{L^2}
\]
\be\label{eqn437}
 \times\| \nabla_\xi  \widehat{f}_{k_2}(t, \xi)\|_{L^2} \big] \lesssim 2^{m -2k+k_1/2+ k_2/2+ 3\min\{k,k_2\}/2+\epsilon  k_{-}+\gamma(k_1+k_2)-2(k_{1,+}+k_{2,+})} \epsilon_1^2.
\ee

 \[
  \|J_{k,k_1,k_2}^{m,1,3}(\xi)\|_{L^2}\lesssim \sup_{t\in[2^{m-1}, 2^m]} 2^{m + \epsilon k_{-} +2k_2-2k+ 3\min\{k,k_2\}/2}\big[ \big(\| \nabla_\xi^2 \widehat{f}_{k_1}(t, \xi)\|_{L^2}\|+ 2^{-k_2} \| \nabla_\xi   \widehat{f}_{k_1}(t, \xi)\|_{L^2}\big)   \]
\[ \times \|f_{k_2}(t)\|_{L^2} +  \big(\| \nabla_\xi^2 \widehat{f}_{k_2 }(t, \xi)\|_{L^2} + 2^{-k_2} \| \nabla_\xi  \widehat{f}_{k_2}(t, \xi)\|_{L^2} + 2^{-2k_2} \| f_{k_2}(t)\|_{L^2}\big) \|  f_{k_1}(t)\|_{L^2}   +  \| \nabla_\xi  \widehat{f}_{k_1}(t, \xi)\|_{L^2}	\]
\be\label{eqn438}
 \times \| \nabla_\xi  \widehat{f}_{k_2}(t, \xi)\|_{L^2}\big]\lesssim 2^{m -2k+k_1/2+ k_2/2+ 3\min\{k,k_2\}/2 +\epsilon  k_{-}+\gamma(k_1+k_2)-2(k_{1,+}+k_{2,+})} \epsilon_1^2.
\ee
Combine estimates  (\ref{eqn320}), (\ref{eqn390}), (\ref{eqn401}), (\ref{eqn402}),  (\ref{eqn420}),(\ref{eqn423}), (\ref{eqn437})  and (\ref{eqn438}), it is easy to see our desired estimate (\ref{eqn410}) holds.

\end{proof}
\begin{lemma}\label{L2weighttwo}
Under the bootstrap assumption \textup{(\ref{smallness})}, the following estimates hold for $k_2 \leq k_1-10$,
\[
\| 	I_{k,k_1,k_2}^{m,3}(\xi) \|_{L^2}  \lesssim 2^{-k/2 +\epsilon  k_{-}+\gamma(k_1+k_2)-2(k_{1,+}+k_{2,+})}
\]
\be\label{eqn760}
 \times \min\{2^{m +2k_2  } \big(1+2^{(m+2k_1) } \big) ,2^{-m/2+  k_2/2- 3k/2 }\}\epsilon_1^2. 
 \ee
 \[
 \| J_{k,k_1,k_2}^{ m,2,1}(\xi) \|_{L^2} + \| J_{k,k_1,k_2}^{ m,2,2}(\xi) \|_{L^2} +\| J_{k,k_1,k_2}^{ m,2,3}(\xi) \|_{L^2} \lesssim 2^{-3k/2+ \epsilon  k_{-}+\gamma(k_1+k_2)-2(k_{1,+}+k_{2,+})}
 \]
 \be\label{eqn879}
\big(  \min\{ 2^{-m/2-k_2 }, 2^{m +2k_2  }\}  + \min\{ 2^{-m/2-k_2/2-k/2 }, 2^{m+k_2+k  }\}\big)\epsilon_1^2.
 \ee
\end{lemma}
\begin{proof}
Note that $|k-k_1|\leq 10$.
 Recall (\ref{eqn240}) and (\ref{eqn271}). From the $L^2-L^\infty$ type bilinear estimate and the $L^\infty\rightarrow L^2$ type Sobolev embedding, the following estimates hold,
 \[
 \| I_{k,k_1,k_2}^{m,3}(\xi)\|_{L^2} \lesssim \sup_{t\in [2^{m-1},2^m]} 2^{m +\epsilon  k_{-} }(1+2^{m+2k_1})\big( \| \nabla_\xi\widehat{f_{k_1}}(t, \xi)\|_{L^2} + 2^{-k_1}\| \  {f_{k_1}}(t )\|_{L^2} \big) \| e^{ it \Delta} f_{k_2}(t)\|_{L^\infty} 
 \]
 \be\label{eqn251}
 \lesssim 2^{m +2k_2 -k_1/2 + \epsilon  k_{-}+\gamma(k_1+k_2)-2(k_{1,+}+k_{2,+})}(1+2^{m+2k_1})\epsilon_1^2.
 \ee
\[
  \| J_{k,k_1,k_2}^{ m,2,1}(\xi)\|_{L^2} \lesssim  \sup_{t\in [2^{m-1},2^m]}2^{m +\epsilon k_{ -}} \| e^{ it \Delta} f_{k_2}(t)\|_{L^\infty}   ( \| \nabla_\xi^2\widehat{f_{k_1}}(t, \xi)\|_{L^2} +2^{-k_1} \| \nabla_\xi \widehat{f_{k_1}}(t, \xi)\|_{L^2}
\]
 \be\label{eqn430}
 + 2^{-2k_1}\| \  {f_{k_1}}(t )\|_{L^2}) \lesssim 2^{-3k/2 + \epsilon  k_{-}+\gamma(k_1+k_2)-2(k_{1,+}+k_{2,+})}\min\{2^{m + 2k_2  } , 2^{-m/2-k_2  }\}\epsilon_1^2.
 \ee

 Note that the following estimate holds when $|\eta|\leq 2^{-5}|\xi|$,
\[
|\xi\cdot (\xi-\eta)|\sim |\xi|^2\sim |\xi-\eta|^2.
\]
Hence, we take the advantage of high oscillation in time by doing integration by parts in time once for $I_{k,k_1,k_2}^{m,3}(\xi)$. As a result, we have
\[
I_{k,k_1,k_2}^{m,3}(\xi):=\int_{t_1}^{t_2} \int_{\R^3} e^{i 2s \xi \cdot(\xi-\eta)}  \frac{ i \psi_k(\xi)}{2 \xi \cdot \xi-\eta} \p_s \Big[\p_{\xi_l}\big( q(\eta,\xi-\eta )   \widehat{\bar{f}_{k_1}}(t, \xi-\eta)\big)\widehat{  {f}_{k_2}}(t, \eta)   
\]
\[
+i2s (2\xi_l-\eta_l)  q(\eta,\xi-\eta )\widehat{\bar{f}_{k_1}}(t, \xi-\eta) \widehat{  {f}_{k_2}}(t, \eta) \Big] d \eta d s + \sum_{i=1,2} (-1)^{i-1}\int_{\R^3} e^{i 2 t_i \xi \cdot(\xi-\eta)}  \frac{i \psi_k(\xi)}{2\xi \cdot \xi-\eta} 
\]
\[
\times\Big[\p_{\xi_l}\big( q(\eta,\xi-\eta )   \widehat{\bar{f}_{k_1}}(t_i, \xi-\eta)\big)\widehat{  {f}_{k_2}}(t_i, \eta)   +i2t_i (2\xi_l-\eta_l)  q(\eta,\xi-\eta )\widehat{\bar{f}_{k_1}}(t_i, \xi-\eta) \widehat{  {f}_{k_2}}(t_i, \eta) \Big] d \eta .
\]
From the $L^2-L^\infty$ type bilinear estimate, (\ref{eqn230}) in Lemma \ref{L2derivatieveestimate} and    (\ref{eqn210}) in Lemma \ref{fixedtimeL2}, we have
\[
\|I_{k,k_1,k_2}^{m,3}(\xi)\|_{L^2} \lesssim 2^{-2k+\epsilon k_{-}}\big[\big( \| \nabla_\xi \widehat{f_{k_1}}(t, \xi)\|_{L^2}  + 2^{-k_1}\| f_{k_1}(t)\|_{L^2} \big) \big( 2^{m }\| e^{ it \Delta} \p_t f_{k_2}(t)\|_{L^\infty} \]
\[+ \| e^{ it \Delta} f_{k_2}\|_{L^\infty}\big) +  2^{m}\big(\| \p_t \nabla_\xi \widehat{f_{k_1}}(t, \xi)\|_{L^2}  + 2^{-k_1}\| \p_t f_{k_1}\|_{L^2}\big) \| e^{ it \Delta} f_{k_2}\|_{L^\infty}  \]
\[
+  (2^{m+k }\| e^{ it \Delta} f_{k_1}\|_{L^\infty} +2^{2m+k}\| e^{ it \Delta} \p_t f_{k_1}\|_{L^\infty} ) \| f_{k_2}\|_{L^2} + 2^{2m+k_1}\| e^{ it \Delta} f_{k_1}\|_{L^\infty} \| \p_t f_{k_2}\|_{L^2}
\big]
\]
\be\label{eqn257}
\lesssim   2^{-m/2+  k_2/2- 2k+ \epsilon  k_{-}+\gamma(k_1+k_2)-2(k_{1,+}+k_{2,+})}\epsilon_1^2.
\ee
Combine estimates (\ref{eqn251}) and (\ref{eqn257}), it is easy to see our desired estimate (\ref{eqn760}) holds.

Recall (\ref{eqn272}) and (\ref{eqn274}).  Very similarly, we also do integration by parts in time once for  $J_{k,k_1,k_2}^{m,2,2}(\xi)$  and  $J_{k,k_1,k_2}^{m,2,3}(\xi)$. After that, for  $J_{k,k_1,k_2}^{m,2,3}(\xi)$,  we do integration by parts in $\eta$ once. As a result, we have 
\[
J_{k,k_1,k_2}^{ m,2,2}(\xi):=\sum_{(l_1,n_1) =(l,n),(n,l) }\int_{t_1}^{t_2} \int_{\R^3} e^{i 2s \xi \cdot(\xi-\eta)} \frac{ -   (2\xi_{n_1}-\eta_{n_1}) }{\xi \cdot(\xi-\eta)  }  \p_s \big(s\p_{\xi_{l_1}}  \big( q(\eta,\xi-\eta)   \widehat{\bar{f}_{k_1}}(t, \xi-\eta)\big)\]
\[
  \times \widehat{f_{k_2}}(t, \eta)\big) \psi_k(\xi) d\eta d s + \sum_{i=1,2} (-1)^i \int_{\R^3} e^{i 2t_i \xi \cdot(\xi-\eta)} \frac{   t_i(2\xi_{n_1}-\eta_{n_1}) \psi_k(\xi)}{\xi \cdot (\xi-\eta) } \]
  \[\times   \widehat{{f}_{k_2}}(t_i, \eta)   \p_{\xi_{l_1}}  \big( q(\eta,\xi-\eta) \widehat{\bar{f}_{k_1}}(t_i, \xi-\eta) \big)    d\eta,
\]
\[
J_{k,k_1,k_2}^{ m,2,3}(\xi):=\sum_{i=1,2} J_{k,k_1,k_2}^{ m,2,3;i}(\xi), \quad  J_{k,k_1,k_2}^{ m,2,3;1}(\xi):= - \int_{t_1}^{t_2} \int_{\R^3} e^{i 2s \xi \cdot(\xi-\eta)}  \psi_k(\xi) \]
\[\times \frac{\xi}{|\xi|^2}\cdot \p_s\Big(  s \nabla_\eta \big(  \widehat{\bar{f}_{k_1}}(t, \xi-\eta) \frac{   (2\xi_{l}-\eta_{l})(2\xi_n -\eta_n)  q(\eta,\xi-\eta)}{\xi \cdot(\xi-\eta) }   \big) \widehat{{f}_{k_2}}(t, \eta) \Big)   d\eta d s\]
\[
 + \sum_{i=1,2} (-1)^i \int_{\R^3} e^{i 2 t_i \xi \cdot\eta} \psi_k(\xi) t_i  \widehat{{f}_{k_2}}(t_i, \eta) \frac{\xi}{|\xi|^2}  \cdot  \nabla_\eta \Big(  \frac{   (2\xi_{l}-\eta_{l})(2\xi_n -\eta_n)  q(\eta,\xi-\eta)}{\xi \cdot(\xi-\eta) } \widehat{\bar{f}_{k_1}}(t_i, \xi-\eta) \Big) 
    d\eta ,
\]
\[
J_{k,k_1,k_2}^{ m,2,3;2}(\xi):= - \int_{t_1}^{t_2} \int_{\R^3} e^{i 2s \xi \cdot(\xi-\eta)}  \psi_k(\xi) \frac{\xi}{|\xi|^2}\cdot \p_s\Big( s \nabla_\eta  \widehat{{f}_{k_2}}(t, \eta)     \frac{  (2\xi_{l}-\eta_{l})(2\xi_n -\eta_n)  q(\eta,\xi-\eta)}{\xi \cdot(\xi-\eta) }\]
\[ \times \widehat{\bar{f}_{k_1}}(t, \xi-\eta)   \Big)   d\eta d s + \sum_{i=1,2} (-1)^i \int_{\R^3} e^{i 2 t_i \xi \cdot\eta} \psi_k(\xi) t_i \widehat{\bar{f}_{k_1}}(t_i, \xi-\eta)  \frac{\xi}{|\xi|^2}  \cdot \nabla_\eta \widehat{{f}_{k_2}}(t_i, \eta) 
 \]
\be\label{eqn285} 
\times  \frac{   (2\xi_{l}-\eta_{l})(2\xi_n -\eta_n)  q(\eta,\xi-\eta)}{\xi \cdot(\xi-\eta) }  
    d\eta .
\ee
From the $L^2-L^\infty$ type bilinear estimate, the $L^\infty\rightarrow L^2$ type Sobolev embedding, (\ref{eqn230}) and (\ref{e300}) in Lemma \ref{L2derivatieveestimate}, and (\ref{eqn210}) in Lemma \ref{fixedtimeL2}, we have
\[
\| J_{k,k_1,k_2}^{ m,2,2}(\xi)\|_{L^2} +\| J_{k,k_1,k_2}^{ m,2,3;1}(\xi)\|_{L^2}  \lesssim 2^{ m-k+\epsilon k_{-}}\big[ \big( 2^{m} \| \p_t  \nabla_\xi \widehat{f_{k_1}}(t, \xi)\|_{L^2} + \|\nabla_\xi \widehat{f_{k_1}}(t, \xi)\|_{L^2}   \big) \| e^{ it\Delta} f_{k_2}\|_{L^\infty}\]
\[
 + 2^{-k_2}  \big( \| f_{k_2}(t)\|_{L^2} +   2^{m} \| \p_t f_{k_2}(t)\|_{L^2}\big) \min\{\| e^{ it \Delta} f_{k_1}(t)\|_{L^\infty} + 2^m \| e^{ it \Delta} \p_t f_{k_1}(t)\|_{L^\infty}  , \]
\[  2^{3k_2/2} \big(\| f_{k_1}\|_{L^2}  + 2^m\| \p_t f_{k_1}\|_{L^2} \big)  \} \big]\lesssim  2^{-3k/2+  \epsilon  k_{-}+\gamma(k_1+k_2)-2(k_{1,+}+k_{2,+})}\epsilon_1^2 
\]
\be\label{eqn429}
\times \big(  \min\{ 2^{-m/2-k_2 }, 2^{m +2k_2  }\}  + \min\{ 2^{-m/2-k_2/2-k/2 }, 2^{m+k_2+k  }\}\big),
\ee

\[
\| J_{k,k_1,k_2}^{ m,2,3;2}(\xi)\|_{L^2} \lesssim 2^{m-k+\epsilon k_{-} }\big[ \big( 2^m \| \p_t \nabla_\xi \widehat{f}_{k_2}(t, \xi) \|_{L^2} + \| \nabla_\xi \widehat{f}_{k_2}(t, \xi)\|_{L^2} \big)  \min\{\| e^{ it \Delta} f_{k_1}(t)\|_{L^\infty}\]
\[
    +  2^{m} \| e^{ it \Delta} \p_t f_{k_1}\|_{L^\infty}, 2^{3k_2/2}\big(\| f_{k_1}(t)\|_{L^2} + 2^m \| \p_t f_{k_1}(t)\|_{L^2} \big) \} \big]
\]
\be\label{eqn428}
\lesssim  2^{-3k/2+\epsilon  k_{-}+\gamma(k_1+k_2)-2(k_{1,+}+k_{2,+})}\min\{ 2^{-m/2-k_2/2-k/2}, 2^{m+k_2+k } \}\epsilon_1^2.
\ee
 Combine (\ref{eqn430}), (\ref{eqn429}), and (\ref{eqn428}), it is easy to see our desired estimate (\ref{eqn879}) holds.

\end{proof}
\begin{lemma}\label{fixedtimeL2}
Under the bootstrap assumption \textup{(\ref{smallness})}, the following estimate holds for any $k\in \mathbb{Z}$,
\be\label{eqn210}
\| \p_t \nabla_\xi \widehat{f}(t, \xi)\psi_k(\xi)\|_{L^2} \lesssim 2^{-m -k/2+\gamma k-2k_{+}}\epsilon_1^2. 
\ee
\end{lemma}
\begin{proof}
 Note that the following equality holds for any $l\in \{1,2\},$
\[
\p_{\xi_l}\p_t \widehat{f }(t, \xi)\psi_k(\xi)=  \sum_{(k_1,k_2)\in \chi_k^1 \cup \chi_k^2 }  \sum_{i=1,2} F_{k,k_1,k_2}^i, \quad  F_{k,k_1,k_2}^1 =   \int_{\R^3}\big( e^{i 2 t\xi \cdot \eta}   \widehat{\bar{f}_{k_2}}(t, \eta) \]
\be\label{eqn491}
\times  \p_{\xi_l}\big( q(\xi-\eta,\eta)  \widehat{f_{k_1}}(t, \xi-\eta)\big)+ e^{i 2 t\xi \cdot (\xi-\eta) } \p_{\xi_l}\big( q(\eta,\xi-\eta)  \widehat{\bar{f}_{k_1}}(t, \xi-\eta)\big) \widehat{ {f}_{k_2}}(t, \eta)\big) \psi_k(\xi) d \eta, 
\ee
\[F_{k,k_1,k_2}^2 =    \int_{\R^3}   i2 t \big(e^{i 2 t\xi \cdot \eta} \eta_l q(\xi-\eta,\eta)  \widehat{f_{k_1}}(t, \xi-\eta)  \widehat{\bar{f}_{k_2}}(t, \eta)\]
\be\label{eqn492}
 +e^{i 2 t\xi \cdot (\xi-\eta)} (2\xi_l-\eta_l) q(\eta,\xi-\eta )  \widehat{\bar{f}_{k_1}}(t, \xi-\eta)  \widehat{  {f}_{k_2}}(t, \eta)  \big) \psi_k(\xi) d \eta.
\ee
From the $L^2-L^\infty$ type bilinear estimate (\ref{bilinearestimate}) in Lemma \ref{bilinearlemma} and  the  $L^\infty\rightarrow L^2$ type Sobolev embedding, the following estimate holds,
\[
\sum_{i=1,2} \|  F_{k,k_1,k_2}^i\|_{L^2 } \lesssim 2^{ \epsilon  k_{-}+ 3\min\{k,k_2\}/2}\big( \| \nabla_\xi \widehat{f_{k_1}}(t, \xi)\|_{L^2} + 2^{-k} \| f_{k_1}(t)\|_{L^2} + 2^{m+k_1} \| f_{k_1}(t)\|_{L^2} \big) \| f_{k_2}(t)\|_{L^2} 
\]
\be\label{eqn204}
\lesssim 2^{ 3\min\{k,k_2\}/2+ k_2/2-k+k_1/2+ \epsilon  k_{-}+\gamma(k_1+k_2)-2(k_{1,+}+k_{2,+}) }(1+2^{m+k+k_1})\epsilon_1^2 .
\ee
 
On the other hand, after doing integration by parts in $\eta$ once for $ F_{k,k_1,k_2}^1$ and doing integration by parts in $\eta$ twice for $ F_{k,k_1,k_2}^2 $, we have
\[
F_{k,k_1,k_2}^1 =   \int_{\R^3} e^{i 2 t\xi \cdot \eta} \frac{i\xi}{2t |\xi|^2}\cdot\nabla_\eta \big( \p_{\xi_l}\big( q(\xi-\eta,\eta)  \widehat{f_{k_1}}(t, \xi-\eta)\big) \widehat{\bar{f}_{k_2}}(t, \eta)\big) \psi_k(\xi) d \eta,
\]
\[
+ \int_{\R^3} e^{i 2 t\xi \cdot (\xi-\eta)} \frac{i\xi}{2t |\xi|^2}\cdot\nabla_\eta \big( \p_{\xi_l}\big( q(\eta, \xi-\eta )  \widehat{\bar{f}_{k_1}}(t, \xi-\eta)\big) \widehat{  {f}_{k_2}}(t, \eta)\big) \psi_k(\xi) d \eta,
\]
\[
F_{k,k_1,k_2}^2 =     \int_{\R^3} e^{2i t\xi \cdot \eta}\frac{-\xi}{4t|\xi|^2}\cdot \nabla_\eta \Big( \frac{\xi}{|\xi|^2}\cdot \nabla_\eta \big( \eta_{l}  q(\xi-\eta,\eta)  \widehat{f_{k_1}}(t, \xi-\eta) \widehat{\bar{f}_{k_2}}(t, \eta) \big)\Big)\psi_k(\xi) d \eta
\]
\[
 + \int_{\R^3} e^{2i t\xi \cdot (\xi-\eta)}\frac{-\xi}{4t|\xi|^2}\cdot \nabla_\eta \Big( \frac{\xi}{|\xi|^2}\cdot \nabla_\eta \big( (2\xi_l-\eta_l) q(\eta,\xi-\eta )  \widehat{\bar{f}_{k_1}}(t, \xi-\eta)  \widehat{  {f}_{k_2}}(t, \eta)  \big)\Big)\psi_k(\xi) d \eta.
\]
Therefore, from  the $L^2-L^\infty$ type and the $L^4-L^4$ type bilinear estimates (\ref{bilinearestimate}) in Lemma \ref{bilinearlemma}, the following estimate holds,
\[
\sum_{i=1,2}\|F_{k,k_1,k_2}^i\|_{L^2}  \lesssim 2^{-m-2k+k_1+\epsilon k_{ -}}\big[(\| \nabla_\xi^2 \widehat{f_{k_2}}(t, \xi)\|_{L^2} + 2^{-k_2}\| \nabla_\xi \widehat{f_{k_2}}(t, \xi)\|_{L^2} +2^{-2k_2}\| f_{k_2}(t)\|_{L^2})
\]
\[
\times \| e^{ it \Delta} f_{k_1}\|_{L^\infty}+   \| e^{ it \Delta} \mathcal{F}^{-1}[\nabla_\xi  \widehat{f}_{k_1}(t, \xi)]\|_{L^4}\big(  2^{-k_2}\| e^{ it \Delta} f_{k_2}(t)\|_{L^4}+ \| e^{ it \Delta} \mathcal{F}^{-1}[\nabla_\xi  \widehat{f}_{k_2}(t, \xi)]\|_{L^4}\big)\big]
\]
\be\label{eqn200}
  \lesssim   2^{-5m/2-3k_2/2-2k +  \epsilon  k_{-}+\gamma(k_1+k_2)-2(k_{1,+}+k_{2,+})}\epsilon_1^2.
\ee
Hence, from (\ref{eqn204}) and (\ref{eqn200}), we have
\[
 \sum_{i=1,2}   \sum_{(k_1,k_2)\in \chi_k^1\cup \chi_k^2 }   \|F_{k,k_1,k_2}^i\|_{L^2}  \lesssim \sum_{|k_1-k_2|\leq 10,  k-10\leq k_1\leq -m -k} 2^{\epsilon k_{-}+ k/2 +k_1 +2\gamma k_1-4k_{1,+}} (1+2^{m+k+k_1}) \epsilon_1^2 \]
 \[
+ \sum_{|k_1-k_2|\leq 10,   k_1\geq \max\{ -m -k, k-10\}} 2^{-5m/2-3k_1/2-2k +\epsilon k_{-}+ 2\gamma k_1 - 4k_{1,+}} \epsilon_1^2 + \sum_{|k_1-k|\leq 10, k_2\leq \min\{-m-k, k_1\}}2^{  \epsilon  k_{-}}
 \]
 \[
 2^{  \gamma(k_1+k_2)-2 k_{1,+}  +2k_2 -k_1/2  } (1+2^{m+2k_1}) \epsilon_1^2+   \sum_{|k_1-k|\leq 10,   -m-k\leq k_2\leq k_1}  2^{-5m/2-3k_2/2-2k }  
 \]
 \[  
 \times 2^{\epsilon  k_{-}+\gamma(k_1+k_2)-2(k_{1,+}+k_{2,+})}\epsilon_1^2\lesssim 2^{-m -k/2+\gamma k-2k_{+}}\epsilon_1^2.  
\]
Hence finishing the proof.
\end{proof}

\subsection{The proof of Theorem \ref{maintheorem} }
Firstly, we do the energy estimate. Recall (\ref{eqn201}). It's easy to see the following estimate holds for any $t\in[0,T]$,
\[
\| u(t)\|_{H^{10}}^2 = \| f(t)\|_{H^{10}}^2\lesssim \epsilon_0^2 + \sum_{k\in \mathbb{Z}} \sum_{(k_1,k_2)\in \chi_k^1 \cup \chi_k^2}  \int_0^{t}  2^{\epsilon k_{-} + 20 k_{+}}  \| f_k(s)\|_{L^2}  \| f_{k_1}(s)\|_{L^2} \| e^{ it \Delta}  f_{k_2}(s)  \|_{L^\infty }  d s 
\] 
\be\label{conclusion2}
\lesssim \epsilon_0^2 + \int_0^t \| f(s)\|_{H^{10}}^2 \frac{\epsilon_1}{(1+s)^{ 1+\gamma/2}} d s \lesssim \epsilon_0^2.
\ee

Lastly, we proceed to estimate the $Z$-norm of the profile $f (t) $.  Recall (\ref{duhamel1}), (\ref{eqn20}), and (\ref{eqn1}).  For any fixed $k\in \mathbb{Z}$, we have
\[
\| f_k(t) \|_{Z}\lesssim  \epsilon_0 + \sum_{1\leq m\leq \log(T)} 2^{-\gamma k +2k_{+}}\big[ 2^{-k/2 } \sum_{k_2\leq k_1, k-3\leq k_1}  \|H_{k,k_1,k_2}^m  \|_{L^2}   
\]
\[
 + \sum_{k_2\leq k_1+10, k-3\leq k_1} \big(\sum_{ i=1,2} 2^{k/2 }\| I^{m ,i}_{k,k_1,k_2}\|_{L^2} +\sum_{i=1,2,3}  2^{3k/2 } \|J_{k,k_1,k_2}^{m,  1,i}\|_{L^2}\big)
\]
\be\label{eqn500}
 + \sum_{k_2\leq k_1-10, k-3\leq k_1 \leq k+3}\big(  2^{k/2  }\| I^{m ,3}_{k,k_1,k_2}\|_{L^2}  + \sum_{i=1,2,3}  2^{3k/2 } \|J_{k,k_1,k_2}^{m,  2,i}\|_{L^2} \big)\big].
\ee

From estimate (\ref{eqn298}) in Lemma \ref{L2derivatieveestimate}, estimate  (\ref{eqn410})   in Lemma \ref{L2weightone}, and estimates (\ref{eqn760}) and (\ref{eqn879}) in Lemma \ref{L2weighttwo}, we have
\[
\textup{(\ref{eqn500})} \lesssim \epsilon_0+\sum_{1\leq m\leq \log(T)}\Big[\sum_{k\leq k_1+10, |k_1-k_2| \leq 10} 2^{2\gamma  k_1 +(\epsilon-\gamma) k -2k_{1,+}} \epsilon_1^2 \min\{ 2^{m+k+k_1}, 2^{-m/2-k/2-k_1/2 } \}
\]
\[   + \sum_{k_2\leq k -10 } 2^{\gamma k_2-2k_{2,+}+\epsilon k_{-}} \epsilon_1^2\big(\min\{2^{m+2k_2  }, 2^{-m/2-k_2 } \} +  \min\{ 2^{-m/2-k_2/2-k/2 }, 2^{m+k_2+k  }(1+2^{m+k_2+k  })\} \big)\Big]
\]
\[
\lesssim  \epsilon_0+ \sum_{1\leq m \leq \log(T)} \epsilon_1^2 \big[ \sum_{k_1 \geq \max\{-m-k,k-10\}} 2^{-m/2-k/2-k_{1}/2+2\gamma k_{1}+(\epsilon-\gamma)k -2k_{1,+}}\]
\[ + \sum_{k-10\leq k_1\leq -m -k} 2^{m+  (1+2\gamma) k_{1 } + (1+\epsilon-\gamma)k -2k_{1,+}}+ \sum_{k_2\geq -m-k} 2^{-m/2-k_2/2-k/2 +\epsilon k_{-}+ \gamma k_2-2k_{2,+}}
\]
\[
  + \sum_{k_2 \leq -m-k} 2^{m+k_2+k+\epsilon k_{-}+\gamma k_2-2k_{2,+}} (1+2^{m+k_2+k}) + \sum_{k_2 \geq -m/2} 2^{-m/2 -k_2 +\epsilon k_{-}+ \gamma k_2-2k_{2,+}} \]
\be\label{conclusion}
+ \sum_{k_2\leq -m/2} 2^{m+2k_2 +\epsilon k_{-} +  \gamma k_2-2k_{2,+}}\big]\lesssim \epsilon_0+ \sum_{1\leq m \leq \log(T)} 2^{-\gamma m/2}\epsilon_1^2 + 2^{-(\epsilon-\gamma) m}\epsilon_1^2  \lesssim \epsilon_0.
\ee
Hence finishing the proof. 
\begin{remark}
If we let $\epsilon=\gamma=0$,  note that above argument also works. From (\ref{conclusion2}) and (\ref{conclusion}), we have
\[
\| f(t)\|_{H^{10}}+ \| f(t)\|_{Z} \lesssim \epsilon_0 + \log T \epsilon_1^2.
\]
Hence the bootstrap argument still works out as long as $T \leq e^{c/\epsilon_0}$, where $c$ is a very small constant. Therefore, the method we developed here also proves the small data almost global existence of the $3D$ quadratic Schr\"odinger as follows, $(\p_t -i \Delta)u = |u|^2.$
 
\end{remark}

\begin{remark}
The main issue that prevents us to get the small data global existence of the $3D$ quadratic Schr\"odinger equation is the accumulated effect over time of the interaction of $\sqrt{t}\times \sqrt{t}\longrightarrow \sqrt{t}$ type. More precisely, both the input frequencies and the output frequencies are of size $\sqrt{t}$.	Therefore, certainly the upper bound $2^{\epsilon k_{-}}$ in  the assumption (\ref{assumption}) can be relaxed to $2^{\epsilon k_{1,-}}$. But it is not clear how to completely remove this assumption at this point. 
\end{remark}

\end{document}